\documentclass{amsart}
\usepackage[utf8]{inputenc}

\usepackage{amsfonts}
\usepackage{amsmath}
\usepackage{amsthm}
\usepackage{amssymb}
\usepackage{bm}
\usepackage{subcaption}

\usepackage{color}
\usepackage[dvipsnames]{xcolor}

\usepackage{graphicx}\graphicspath{{./images/}}
\usepackage{tikz}\tikzset{every picture/.style=semithick}
\usetikzlibrary{shapes,arrows,decorations.markings,decorations.pathreplacing}
\tikzset{->-/.style={decoration={
    markings,
    mark=at position #1 with {\arrow{angle 60}}},postaction={decorate}}}

\usepackage{hyperref}
\usepackage[noabbrev,nameinlink,capitalise]{cleveref}
\usepackage[style=alphabetic]{biblatex}
\theoremstyle{plain}
    \newtheorem{theorem}{Theorem}[section]
    \newtheorem{corollary}[theorem]{Corollary}
    \newtheorem{lemma}[theorem]{Lemma}
    \newtheorem{proposition}[theorem]{Proposition}
    
\theoremstyle{definition}
    \newtheorem{definition}[theorem]{Definition}
    
\theoremstyle{remark}
    \newtheorem{remark}[theorem]{Remark}

\title[A class of rearrangement groups that are not IG]{A Class of Rearrangement Groups that are not Invariably Generated}

\author{Davide Perego \and Matteo Tarocchi}

\date{}


\thanks{Both the authors are members of the Gruppo Nazionale per le Strutture Algebriche, Geometriche e le loro Applicazioni (GNSAGA) of the Istituto Nazionale di Alta Matematica (INdAM)}

\address{Dipartimento di Matematica e Applicazioni, Universit\`a degli Studi di Milano-Bicocca, Ed. U5, Via R.Cozzi 55, 20125 Milano, Italy, EU}
\curraddr{Instituto de Matemáticas, Universidad de Sevilla, Ed. Celestino Mutis, Campus de Reina Mercedes, Avda. Reina Mercedes s/n, 41012 Sevilla, Spain, EU}
\email{\href{mailto:dperego9@gmail.com}{dperego9@gmail.com}}

\address{Dipartimento di Matematica e Applicazioni, Universit\`a degli Studi di Milano-Bicocca, Ed. U5, Via R.Cozzi 55, 20125 Milano, Italy, EU}
\email{\href{mailto:matteo.tarocchi.math@gmail.com}{matteo.tarocchi.math@gmail.com}}

\begin{document}

\begin{abstract}
    A group $G$ is invariably generated if there exists a subset $S \subseteq G$ such that, for every choice $g_s \in G$ for $s \in S$, the group $G$ is generated by $\{ s^{g_s} \mid s \in S \}$.
    In \cite{Gelander2016InvariableGO} Gelander, Golan and Juschenko showed that Thompson groups $T$ and $V$ are not invariably generated.
    Here we generalize this result to the larger setting of rearrangement groups, proving that any subgroup of a rearrangement group that has a certain transitive property is not invariably generated.
\end{abstract}

\maketitle

\section{Introduction}
\label{SEC intro}

A group $G$ is \textbf{invariably generated} (\textbf{IG} for short) if there exists a subset $S \subseteq G$ such that, for every choice $g_s \in G$ for $s \in S$, the group $G$ is generated by $\{ s^{g_s} \mid s \in S \}$ (with $g^h$ we mean $h^{-1} g h$).
In other words, $G$ is invariably generated if there exists a generating set $S$ such that, if one modifies $S$ by conjugating each $s \in S$ by some $g_s \in G$ (possibly a distinct $g_s$ for each $s$), then they still end up with a generating set for $G$.

This notion has been present in literature for a long time, for example in \cite{Wiegold1976TransitiveGW, Wiegold1976TransitiveGW2} under a different but equivalent definition.
However, the term \textit{Invariable Generation} was first used in \cite{DIXON199225} by Dixon in his study about computational Galois theory.
Finite groups are invariably generated, hence interesting questions in that setting generally involve sizes of invariable generating sets (\cite{KANTOR2011302, GARZONI2020218}).
Instead, the context of infinite groups presents both examples of IG and non-IG groups: Houghton groups are IG for all $n \geq 2$ (\cite{COX2022120}), whereas certain convergence groups (in particular hyperbolic groups) and certain arithmetic groups are non-IG (\cite{8200290, 10.1093/imrn/rnw137}).
Moreover, in \cite{Gelander2016InvariableGO} Gelander, Golan and Juschenko proved that Thompson group $F$ is (finitely) invariably generated (later generalized in \cite{FFIG}), while Thompson groups $T$ and $V$ are not invariably generated (see \cite{cfp} for an introduction to these groups).
Since Thompson groups have so many generalizations, it is natural to ask whether invariable generation can be studied for those generalizations using similar methods.

One of the many generalizations of Thompson groups is the class of rearrangement groups, introduced by Belk and Forrest in \cite{belk2016rearrangement}, which are defined as certain groups of homeomorphisms of limit spaces of sequences of graphs.
Many known fractals (such as the Basilica Julia set and the Airplane Julia set) can be realized as such limit spaces.

The class of rearrangement groups contains the Higman-Thompson groups $F_{r,n}$, $T_{r,n}$ and $V_{r,n}$ (\cite{higman1974finitely, BROWN198745}), topological full groups of branching one-sided edge-shifts (\cite{Matui}), certain picture groups (\cite{guba1997diagram}), along with many new groups acting on known fractals, such as the Basilica Thompson group $T_B$ (\cite{Belk_2015}), the Airplane rearrangement group (introduced in \cite{belk2016rearrangement} and studied in \cite{Airplane}) and many others.

The aim of this work is to generalize the results from \cite{Gelander2016InvariableGO} and their proofs to the setting of rearrangement groups, proving the following result:

\medskip
\noindent
\phantomsection\label{THM main intro}
\textbf{Main Theorem 1.1.}
\emph{A CO-transitive subgroup $G$ of a rearrangement group $G_\mathcal{X}$ is not invariably generated.}
\medskip

By \textbf{CO-transitive} (compact-open transitive, \cite{CO-compact}) we mean that $G$ acts on a space $X$ in such a way that, for each proper compact $K$ and each non-empty open $U$ of $X$, there is an element of $G$ that maps $K$ inside $U$.
We will work with a condition (\textit{weak cell-transitivity}, introduced in \cref{SEC transitivity}) that is equivalent to CO-transitivity and easier to verify in the language of replacement systems of graphs.

The Main Theorem has itself a generalization which is worth mentioning:

\medskip
\noindent
\textbf{Corollary 1.2.}
\emph{If $1 \neq N \trianglelefteq G \leq G_\mathcal{X}$, where $G_\mathcal{X}$ is a rearrangement group and $G$ is weakly cell-transitive, then $N$ is not invariably generated.}
\medskip

In particular, many known rearrangement groups are CO-transitive, so we have the following new results:

\medskip
\noindent
\textbf{Corollary 1.3.}
\emph{The following groups (and their commutator subgroups) are not invariably generated:
\begin{itemize}
    \item the Higman-Thompson groups $T_{r,n}$ and $V_{r,n}$ (\cite{higman1974finitely, BROWN198745});
    \item the Basilica Thompson group $T_B$ (\cite{Belk_2015}) and its generalizations (\cite{belk2016rearrangement});
    \item the Airplane rearrangement group $T_A$ (\cite{belk2016rearrangement,Airplane});
    \item the Vicsek rearrangement group and its generalizations (\cite{belk2016rearrangement});
    \item topological full groups of one-sided irreducible branching edge-shifts (\cite{Matui}).
\end{itemize}}
\medskip

We notice that there are currently no known necessary or sufficient conditions to establish whether a subgroup of a rearrangement group is itself a rearrangement group, but that our results give sufficient conditions for non-invariable generation among subgroups of a rearrangement group.

\section{Rearrangement Groups}\label{SEC definitions}

Limit spaces of replacement systems and their rearrangements were introduced in \cite{belk2016rearrangement}, which goes in much more details than we will get to do. In this section we briefly describe the basics of this topic.

\subsection{Replacement systems and limit spaces}\label{SUBSEC limit_spaces}

Essentially, a \textbf{replacement system} consists of a \textbf{base graph} $\Gamma$ colored by the set of colors $\mathsf{C}$, along with a \textbf{replacement graph} $R_c$ for every color $c \in \mathsf{C}$, each with an initial vertex and a terminal vertex.
For example, \cref{fig_replacement_A} depicts the so called Airplane replacement system, denoted by $\mathcal{A}$.
We can expand the base graph $\Gamma$ by replacing one of its edges $e$ with the replacement graph $R_c$ indexed by the color $c$ of $e$, as in \cref{fig_exp_A}.
The graph resulting from this process of replacing one edge with the appropriate replacement graph is called a \textbf{simple expansion}.
Simple expansions can be iterated any finite amount of times, which generate the so-called \textbf{expansions} of the replacement system, such as the one in \cref{fig_exp_A_generic}.

\begin{figure}\centering
\begin{subfigure}{.4\textwidth}
\centering
\begin{tikzpicture}[scale=.8]
    \draw[->-=.5,blue] (-0.5,0) -- node[above]{L} (-2,0) node[black,circle,fill,inner sep=1.25]{}; \draw[->-=.5,blue] (0.5,0) -- node[above]{R} (2,0) node[black,circle,fill,inner sep=1.25]{};
    \draw[->-=.5,red] (0.5,0) node[circle,fill,inner sep=1.25]{} to[out=90,in=90,looseness=1.7] node[above]{T} (-0.5,0);
    \draw[->-=.5,red] (-0.5,0) node[black,circle,fill,inner sep=1.25]{} to[out=270,in=270,looseness=1.7] node[below]{B} (0.5,0) node[black,circle,fill,inner sep=1.25]{};
\end{tikzpicture}
\caption{The base graph $A_1$.}
\label{fig_A_base}
\end{subfigure}
\begin{subfigure}{.5\textwidth}
\centering
\begin{subfigure}{.45\textwidth}
\centering
\begin{tikzpicture}[scale=.8]
    \draw[->-=.5,draw=red] (0,0) node[left]{$v_i$} node[black,circle,fill,inner sep=1.25]{} -- (0,2.1) node[left]{$v_t$} node[black,circle,fill,inner sep=1.25]{};
    
    \draw[-stealth] (0.35,1.05) -- (0.65,1.05);
    
    \draw[->-=.5,red] (1,0) node[black,circle,fill,inner sep=1.25]{} node[black,left]{$v_i$} -- node[right]{a} (1,1.05);
    \draw[->-=.5,red] (1,1.05) -- node[right]{c} (1,2.1) node[black,circle,fill,inner sep=1.25]{} node[black,left]{$v_t$};
    \draw[->-=.5,blue] (1,1.05) node[black,circle,fill,inner sep=1.25]{} -- node[right=.4cm]{b} (2,1.05) node[black,circle,fill,inner sep=1.25]{};
\end{tikzpicture}
\end{subfigure}
\begin{subfigure}{.45\textwidth}
\centering
\begin{tikzpicture}[scale=.8]
    \draw[->-=.5,draw=blue] (0,0) node[black,left]{$v_i$} node[black,circle,fill,inner sep=1.25]{} -- (0,2.1) node[black,left]{$v_t$} node[black,circle,fill,inner sep=1.25]{};
    
    \draw[-stealth] (0.35,1.05) -- (0.65,1.05);
    
    \begin{scope}[shift={(.2,0)}]
    \draw[->-=.5,blue] (1.35,0.7) -- node[right]{e} (1.35,0) node[black,circle,fill,inner sep=1.25]{} node[black,left]{$v_i$};
    \draw[->-=.5,blue] (1.35,1.4) -- node[right]{h} (1.35,2.1) node[black,circle,fill,inner sep=1.25]{} node[black,left]{$v_t$};
    \draw[->-=.5,red] (1.35,1.4) to[out=180,in=180,looseness=1.7] node[left]{f} (1.35,0.7);
    \draw[->-=.5,red] (1.35,0.7) node[black,circle,fill,inner sep=1.25]{} to[out=0,in=0,looseness=1.7] node[right]{g} (1.35,1.4) node[black,circle,fill,inner sep=1.25]{};
    \end{scope}
\end{tikzpicture}
\end{subfigure}
\caption{The two replacement rules: $e \to R_{red}$ if $e$ is red, and $e \to R_{blue}$ if $e$ is blue.}
\label{fig_A_replacement_rule}
\end{subfigure}
\caption{The Airplane replacement system $\mathcal{A}$.}
\label{fig_replacement_A}
\end{figure}

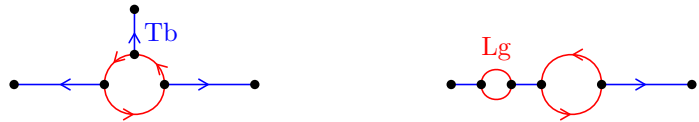
\begin{figure}\centering
\begin{subfigure}{.45\textwidth}\centering
\begin{tikzpicture}[scale=.8]
    \draw[->-=.5,blue] (-0.5,0) -- (-2,0) node[black,circle,fill,inner sep=1.25]{}; \draw[->-=.5,blue] (0.5,0) -- (2,0) node[black,circle,fill,inner sep=1.25]{};
    \draw[->-=.5,red] (0.5,0) to[out=90,in=0] (0,0.5);
    \draw[->-=.5,blue] (0,0.5) -- node[right]{Tb} (0,1.25) node[black,circle,fill,inner sep=1.25]{};
    \draw[->-=.5,red] (0,0.5) node[black,circle,fill,inner sep=1.25]{} to[out=180,in=90] (-0.5,0);
    \draw[->-=.5,red] (-0.5,0) node[black,circle,fill,inner sep=1.25]{} to[out=270,in=270,looseness=1.7] (0.5,0) node[black,circle,fill,inner sep=1.25]{};
\end{tikzpicture}
\end{subfigure}
\begin{subfigure}{.45\textwidth}
\centering
\begin{tikzpicture}[scale=.8]
    \draw[blue] (-0.5,0) -- (-1,0); \draw[blue] (-2,0) node[black,circle,fill,inner sep=1.25]{} -- (-1.5,0);
    \draw[red] (-1,0) to[out=90,in=90,looseness=1.7] node[above]{Lg} (-1.5,0);
    \draw[red] (-1.5,0) node[black,circle,fill,inner sep=1.25]{} to[out=270,in=270,looseness=1.7] (-1,0) node[black,circle,fill,inner sep=1.25]{};
    \draw[->-=.5,blue] (0.5,0) -- (2,0) node[black,circle,fill,inner sep=1.25]{};
    \draw[->-=.5,red] (0.5,0) to[out=90,in=90,looseness=1.7] (-0.5,0);
    \draw[->-=.5,red] (-0.5,0) node[black,circle,fill,inner sep=1.25]{} to[out=270,in=270,looseness=1.7] (0.5,0) node[black,circle,fill,inner sep=1.25]{};
    \node at (0,1.25) {};
\end{tikzpicture}
\end{subfigure}
\caption{Two simple expansions of the base graph of the Airplane replacement system $\mathcal{A}$.}
\label{fig_exp_A}
\end{figure}

\begin{figure}\centering
\begin{tikzpicture}[scale=1.65]
\draw[blue] (-0.5,0) -- (-2,0);
\draw[blue] (0.5,0) -- (2,0);
\draw[blue] (0,-0.5) -- (0,-1.25);
\draw[blue] (0.35,-0.35) -- (0.75,-0.75);
\draw[blue] (0.55,-0.55) -- (0.7,-0.4);
\draw[blue] (1.25,0) -- (1.25,0.75);
\draw[blue] (1.25,0) -- (1.625,0.375);
\draw[blue] (-1.25,0) -- (-1.25,0.75);
\draw[blue] (-1.25,0.5) -- (-1.05,0.5);
\draw[blue] (-1.75,0) -- (-1.75,-0.2);
\draw[red] (0,0) circle (0.5);
\draw[red,fill=white] (-1.25,0) circle (0.25);
\draw[red,fill=white] (1.25,0) circle (0.25);
\draw[red,fill=white] (0.55,-0.55) circle (0.0625);
\draw[red,fill=white] (-1.25,0.5) circle (0.0625);
\draw[red,fill=white] (-1.11875,0.5) circle (0.025);
\draw[red,fill=white] (-1.75,0) circle (0.0625);
\end{tikzpicture}
\caption{A generic expansion of the Airplane replacement system $\mathcal{A}$.}
\label{fig_exp_A_generic}
\end{figure}
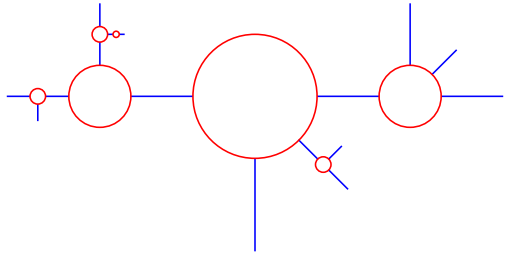

We say that a replacement system is \textit{polychromatic} if $\mathsf{C}$ has more than one element, e.g. the Airplane replacement system $\mathcal{A}$.
Otherwise a replacement system is said to be \textit{monochromatic}, such as the Basilica replacement system (Example 1.2 of \cite{belk2016rearrangement}).

\phantomsection\label{TXT expanding}
A replacement system is said to be \textbf{expanding} if:
\begin{itemize}
    \item neither the base graph nor any replacement graph contains isolated vertices;
    \item in each replacement graph, the initial and terminal vertices are not connected by an edge;
    \item each replacement graph has at least three vertices and two edges.
\end{itemize}

Consider the \textbf{full expansion sequence}, which is the sequence of graphs obtained by replacing, at each step, every edge with the appropriate replacement graph, starting from the base graph.

Whenever a replacement system is expanding, we can define its \textbf{limit space}, which is essentially the limit of the full expansion sequence.
Continuing with our example, the Airplane limit space is depicted in \cref{FIG A limit space}.
Limit spaces have nice topological properties, such as being compact and metrizable (Theorem 1.25 of \cite{belk2016rearrangement}).
Note that then there is no concern over the choice of metric (see Notes 1.30 (1) of \cite{belk2016rearrangement}).

\begin{figure}[b]\centering
\includegraphics[width=.6\textwidth]{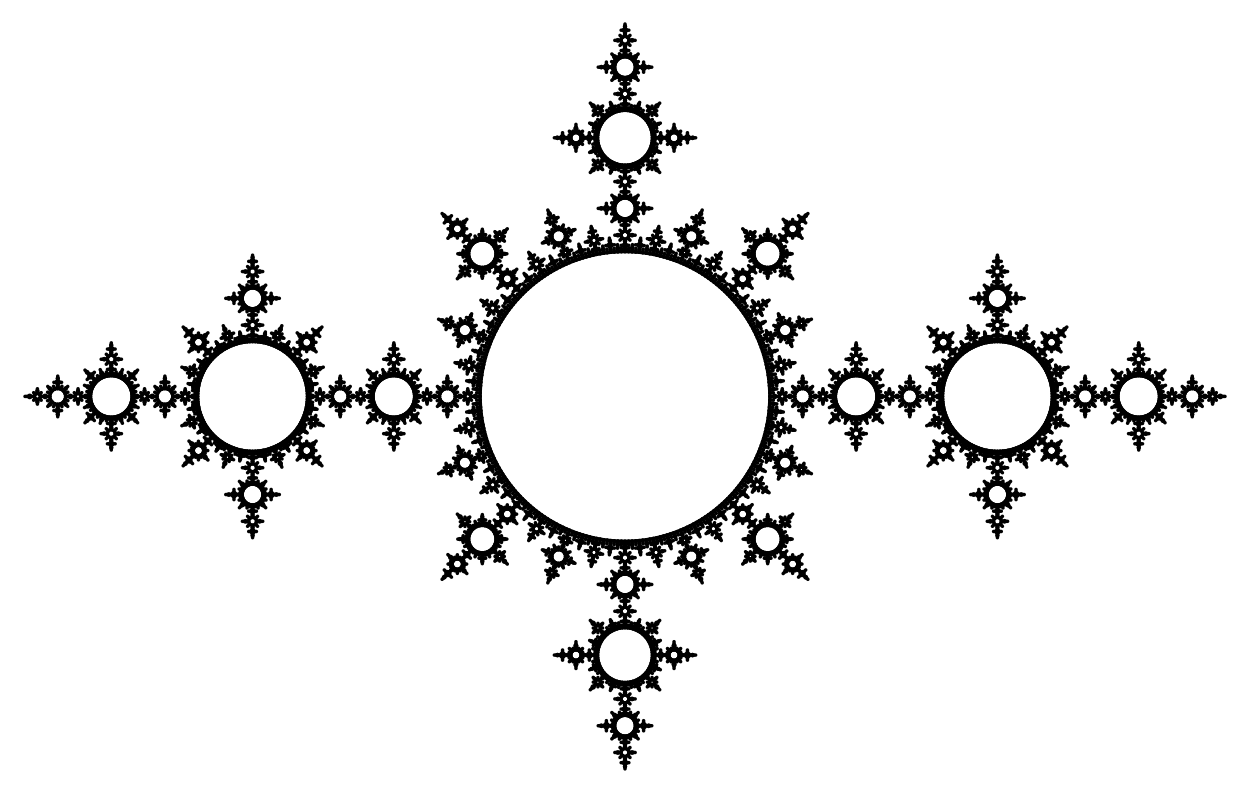}
\caption{The Airplane limit space. Image from \cite{belk2016rearrangement}.}
\label{FIG A limit space}
\end{figure}

\subsection{Cells of a limit space}\label{SUBSEC rearrangements}

Intuitively, a cell $C(e)$ of a limit space corresponds to the edge $e$ of some expansion, along with everything that appears from that edge in later expansions.
\cref{fig_cells_A} shows two examples of cells of $\mathcal{A}$.

\begin{figure}\centering
\includegraphics[width=.425\textwidth]{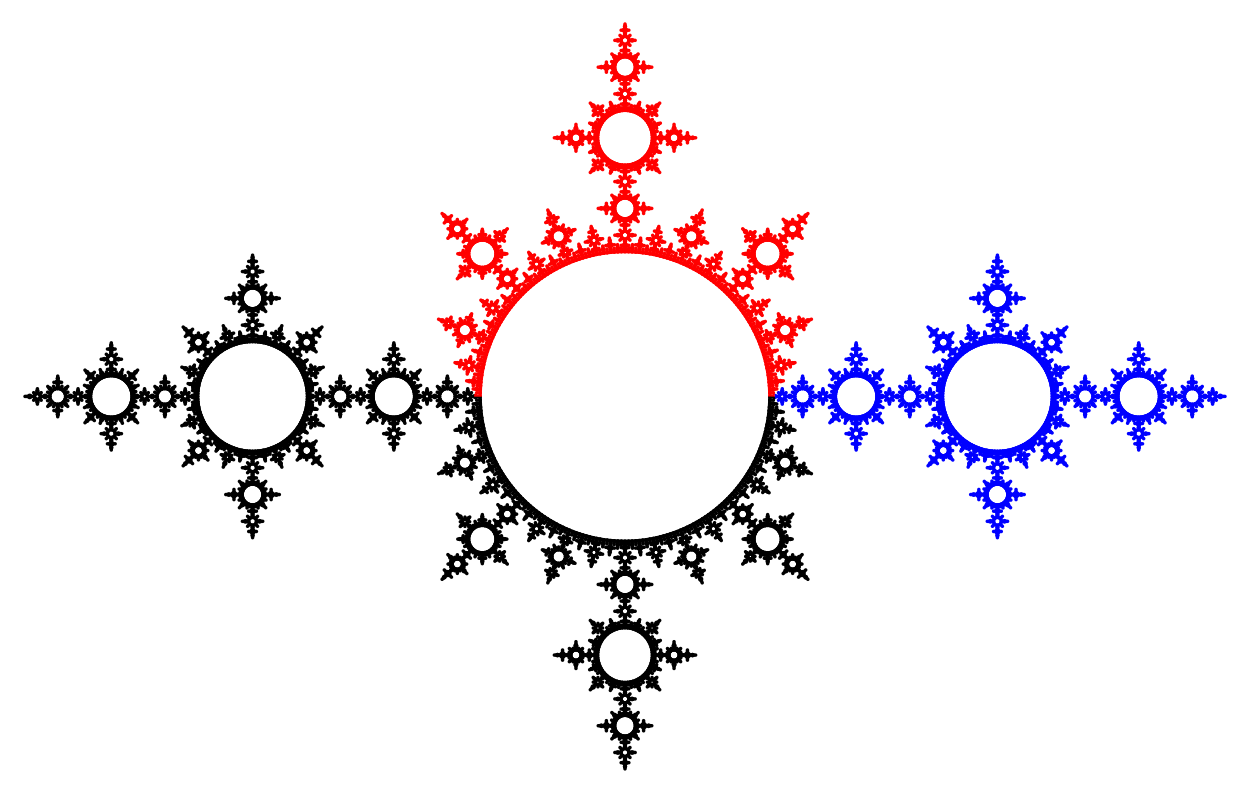}
\caption{The two types of cells in $\mathcal{A}$, distinguished by the color of the generating edge.}
\label{fig_cells_A}
\end{figure}

\phantomsection\label{TXT gluing relation}
Each edge of an expansion of the base graph can be seen as a finite \textit{address} (a sequence of symbols), as in \cref{fig_exp_A}.
Each bit of the address corresponds to an edge of a replacement graph, except for the first one which instead belongs to the base graph.
Points of the limit space are infinite addresses quotiented under a \textit{gluing relation}, which is given by edge adjacency.
More precisely, two sequences $\omega_1 \omega_2 \dots$ and $\omega_1' \omega_2' \dots$ are equivalent when for all $n \in \mathbb{N}$ the prefixes $\omega_1 \dots \omega_n$ and $\omega_1' \dots \omega_n'$ represent edges that share at least a vertex.
With this in mind, since the fibers of the projection are at most countable, it is clear that limit spaces and cells are uncountable.

We call \textit{gluing vertices} of a limit space those points that descend from vertices of some expansion of the base graph, whereas any other point is said to be a \textit{regular point}.

Each cell $C(e)$ has one or two \textit{boundary points}, namely the gluing vertices that descend from the endpoints of the edge $e$.
As defined in \cite{belk2016rearrangement}, the complement of the boundary points is called the \textbf{cell interior}.
We instead denote by $\bm{\mathring{C}(e)}$ the topological interior of the cell $C(e)$.

\begin{remark}
\label{RMK interior}
The topological interior $\mathring{C}(e)$ of the cell $C(e)$ may or may not be the same as the cell interior.
More precisely, the cell interior never includes any of the boundary points, whereas the topological interior includes a boundary point if and only if that point happens to have a unique address, or equivalently it is contained in a unique edge of each expansion of the base graph.
\end{remark}

We would like to highlight the following property regarding cells and balls.

\begin{lemma}
\label{LEM balls cells}
Each ball of a limit space contains a cell and each cell contains a ball.
\end{lemma}

\begin{proof}
Let $B$ be a ball with radius $r > 0$ centered in $p \in X$.
If $p$ is a regular point, $B$ must contain a ball centered at a gluing vertex, because $B$ is open and the set of gluing vertices is dense (since its preimage is dense in the space of infinite addresses), so we can assume that $p$ is a gluing vertex.

Consider the $n$-th star of $p$, denoted by $St_n(p)$, which is the union of $\{p\}$ with the cell interiors of the cells corresponding to the edges of the $n$-th full expansion graph that are incident on $p$.
As is proved in Proposition 1.29 of \cite{belk2016rearrangement}, the sequence of diameters of $St_n(p)$ goes to zero as $n$ approaches infinity, so there exists some $N \in \mathbb{N} = \{0,1,\dots\}$ such that $St_N(p) \subseteq B$.
Since each star clearly contains a cell, $B$ must contain a cell too.

The converse is straightforward, since the cell interior of each cell is open and non-empty.
\end{proof}

\subsection{Rearrangements of limit spaces}

There are different \textit{types} of cells $C(e)$, distinguished by two aspects of the generating edge $e$: its color and whether or not it is a loop.
It is not hard to see that there is a \textbf{canonical homeomorphism} between any two cells of the same type.
A canonical homeomorphism between two cells can essentially be thought as a transformation that maps the first cell ``rigidly'' to the second, or equivalently as a prefix exchange of \hyperref[TXT gluing relation]{addresses} (the much more detailed definition can be found in \cite{belk2016rearrangement}).

\begin{definition}
\label{DEF cellular partition}
A \textbf{cellular partition} of the limit space $X$ is a cover of $X$ by finitely many cells whose cell interiors are disjoint.
\end{definition}

Note that there is a natural bijection between the set of expansions of a replacement system and the set of cellular partitions.

\begin{definition}
\label{DEF rearrangement}
A homeomorphism $f: X \to X$ is called a \textbf{rearrangement} of $X$ if there exists a cellular partition $\mathcal{P}$ of $X$ such that $f$ restricts to a canonical homeomorphism on each cell of $\mathcal{P}$.
\end{definition}

It can be proved that the rearrangements of a limit space $X$ form a group under composition, called the \textbf{rearrangement group} of $X$.

\begin{figure}\centering
\begin{subfigure}{.45\textwidth}
\centering
\begin{tikzpicture}
    \draw[->-=.5] (0,0) -- (4,0);
    \node at (0,0) [circle,fill,inner sep=1.5]{};
    \node at (4,0) [circle,fill,inner sep=1.5]{};
\end{tikzpicture}
\caption{The base graph $\Gamma$.}
\end{subfigure}
\begin{subfigure}{.45\textwidth}
\centering
\begin{tikzpicture}
\draw[-angle 60] (0,0) node[circle,fill,inner sep=1.25]{} node[left]{$v_i$} -- (0,1); \draw (0,1) -- (0,2) node[circle,fill,inner sep=1.25]{} node[left]{$v_t$};

\draw[-stealth] (0.4,1) -- (0.8,1);

\draw[-angle 60] (1.2,0) node[circle,fill,inner sep=1.25]{} node[left]{$v_i$} -- (1.2,0.5); \draw (1.2,0.5) -- (1.2,1) node[circle,fill,inner sep=1.25]{}; \draw[-angle 60] (1.2,1) -- (1.2,1.5); \draw (1.2,1.5) -- (1.2,2) node[circle,fill,inner sep=1.25]{} node[left]{$v_t$};
\end{tikzpicture}
\caption{The replacement rule $e \to R$.}
\end{subfigure}

\caption{The replacement system for the dyadic subdivision of the unit interval $[0,1]$, whose rearrangement group is Thompson group $F$.}
\label{fig_replacement_interval}
\end{figure}
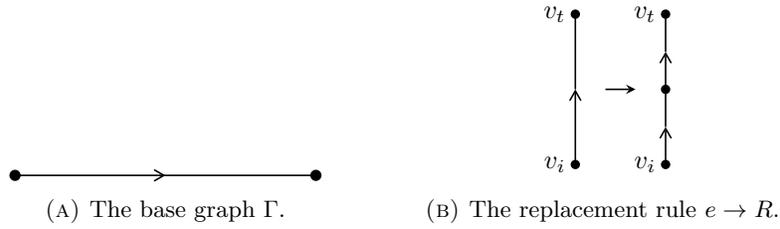

\begin{figure}\centering
\begin{subfigure}{.45\textwidth}
\centering
\begin{tikzpicture}
    \draw[->-=.5] (0,0) circle (1);
    \node at (1,0) [circle,fill,inner sep=1.5]{};
\end{tikzpicture}
\caption{The base graph $\Gamma$.}
\end{subfigure}
\begin{subfigure}{.45\textwidth}
\centering
\begin{tikzpicture}
\draw[-angle 60] (0,0) node[circle,fill,inner sep=1.25]{} node[left]{$v_i$} -- (0,1); \draw (0,1) -- (0,2) node[circle,fill,inner sep=1.25]{} node[left]{$v_t$};

\draw[-stealth] (0.4,1) -- (0.8,1);

\draw[-angle 60] (1.2,0) node[circle,fill,inner sep=1.25]{} node[left]{$v_i$} -- (1.2,0.5); \draw (1.2,0.5) -- (1.2,1) node[circle,fill,inner sep=1.25]{}; \draw[-angle 60] (1.2,1) -- (1.2,1.5); \draw (1.2,1.5) -- (1.2,2) node[circle,fill,inner sep=1.25]{} node[left]{$v_t$};
\end{tikzpicture}
\caption{The replacement rule $e \to R$.}
\end{subfigure}

\caption{The replacement system for the dyadic subdivision of the unit circle $S^1$, whose rearrangement group is Thompson group $T$.}
\label{fig_replacement_circle}
\end{figure}
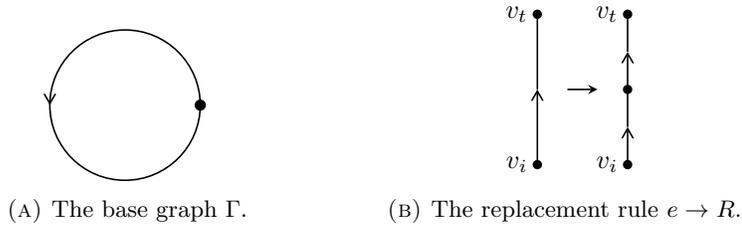

\begin{figure}\centering
\begin{subfigure}{.45\textwidth}
\centering
\begin{tikzpicture}
    \draw[->-=.5] (0,0) -- (4,0);
    \node at (0,0) [circle,fill,inner sep=1.5]{};
    \node at (4,0) [circle,fill,inner sep=1.5]{};
\end{tikzpicture}
\caption{The base graph $\Gamma$.}
\end{subfigure}
\begin{subfigure}{.45\textwidth}
\centering
\begin{tikzpicture}
\draw[-angle 60] (0,0) node[circle,fill,inner sep=1.25]{} node[left]{$v_i$} -- (0,1); \draw (0,1) -- (0,2) node[circle,fill,inner sep=1.25]{} node[left]{$v_t$};

\draw[-stealth] (0.4,1) -- (0.8,1);

\draw[-angle 60] (1.2,0) node[circle,fill,inner sep=1.25]{} node[left]{$v_i$} -- (1.2,0.5); \draw (1.2,0.5) -- (1.2,0.8) node[circle,fill,inner sep=1.25]{}; \draw[-angle 60] (1.2,1.2) node[circle,fill,inner sep=1.25]{} -- (1.2,1.7); \draw (1.2,1.7) -- (1.2,2) node[circle,fill,inner sep=1.25]{} node[left]{$v_t$};
\end{tikzpicture}
\caption{The replacement rule $e \to R$.}
\end{subfigure}

\caption{The replacement system for the Cantor set, whose rearrangement group is Thompson group $V$.}
\label{fig_replacement_cantor}
\end{figure}
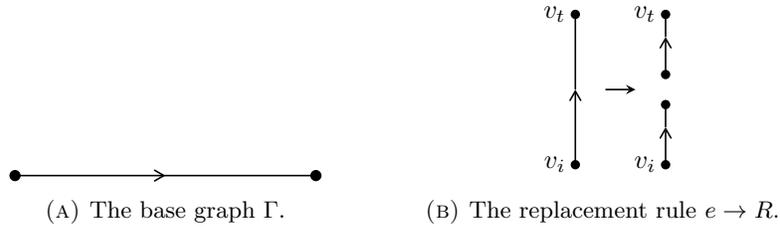

\begin{figure}\centering
\begin{minipage}{.345\textwidth}\centering
\includegraphics[width=\textwidth]{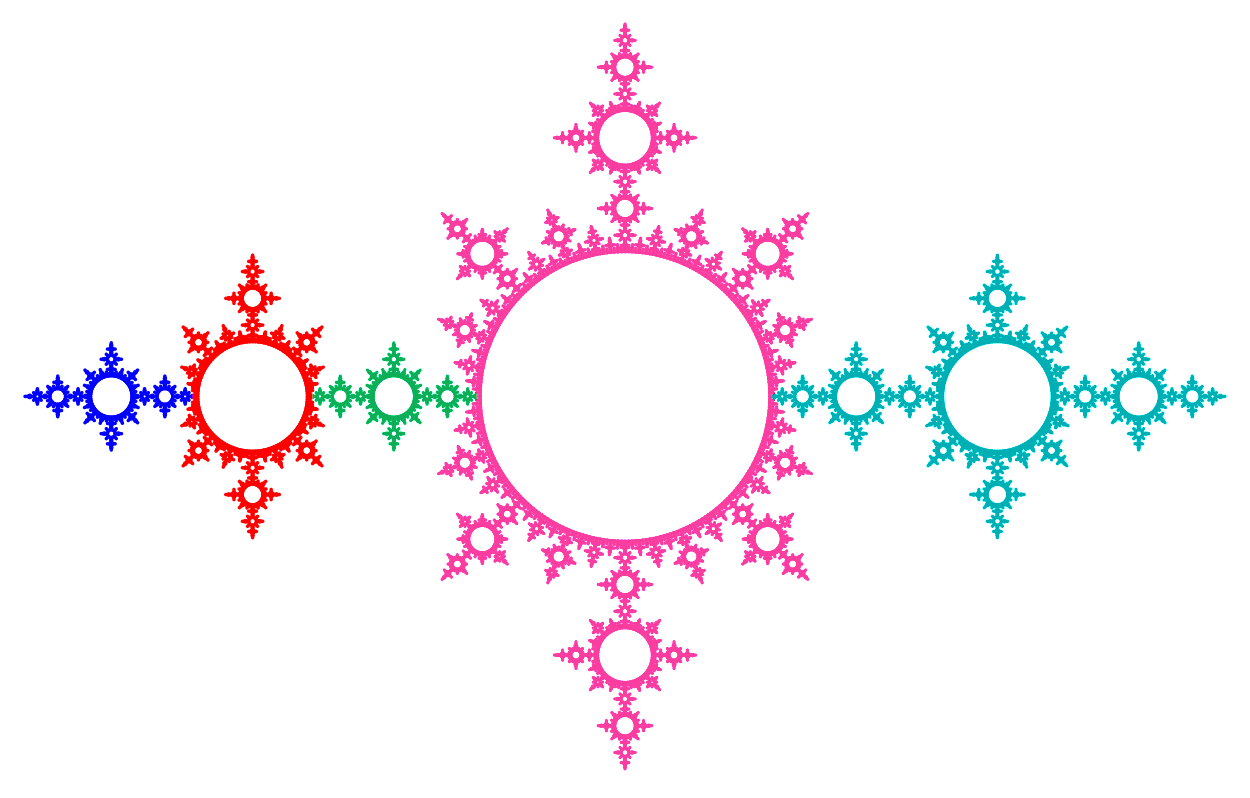}
\end{minipage}%
\begin{minipage}{.05\textwidth}\centering
\begin{tikzpicture}[scale=1]
    \draw[-to] (0,0) -- (.3,0);
\end{tikzpicture}
\end{minipage}%
\begin{minipage}{.345\textwidth}\centering
\includegraphics[width=\textwidth]{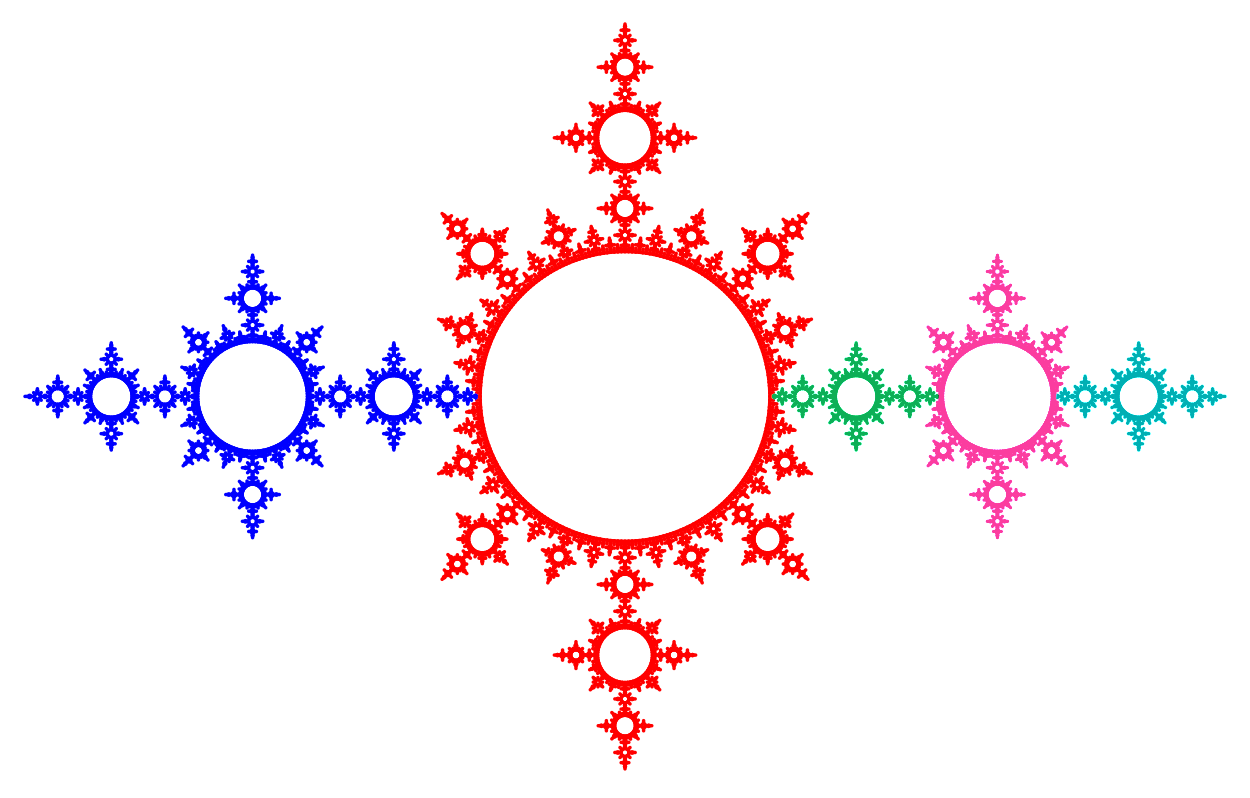}
\end{minipage}\\
\vspace*{10pt}
\begin{tikzpicture}[scale=1]
    \draw[blue] (-2,0) -- (-1.5,0);
    \draw[red] (-1.25,0) circle (0.25);
    \draw[Green] (-1,0) -- (-0.5,0);
    \draw[Green,fill=white] (-0.75,0) circle (0.1);
    \draw[Magenta] (0,0) circle (0.5);
    \draw[TealBlue] (0.5,0) -- (2,0);
    
    \draw[-to] (2.35,0) -- (2.65,0);
    
    \draw[blue] (3,0) -- (4.5,0);
    \draw[red] (5,0) circle (0.5);
    \draw[Green] (5.5,0) -- (6,0);
    \draw[Green,fill=white] (5.75,0) circle (0.1);
    \draw[Magenta] (6.25,0) circle (0.25);
    \draw[TealBlue] (6.5,0) -- (7,0);
\end{tikzpicture}
\caption{A rearrangement of the Airplane limit space, along with a graph pair diagram that represents it.}
\label{fig_action_alpha}
\end{figure}

In the setting of rearrangement groups, the trio of Thompson groups $F$, $T$ and $V$ are realized by the replacement systems in Figures \ref{fig_replacement_interval}, \ref{fig_replacement_circle}, \ref{fig_replacement_cantor}.
Similarly to how dyadic rearrangements (the elements of Thompson's groups) are specified by certain pairs of dyadic subdivisions, rearrangements of a limit space are specified by certain graph isomorphisms between expansions of the replacement systems, called \textbf{graph pair diagrams}.
For example, the rearrangement of the Airplane limit space depicted in \cref{fig_action_alpha} is specified by the graph isomorphism depicted in the same figure.
Colors here mean that each edge of the domain graph is mapped to the edge of the same color in the range graph.

Graph pair diagrams can be \textit{expanded} by expanding an edge in the domain graph and its image in the range graph, resulting in a new graph pair diagram that represents the same rearrangement.
It is important to note that, for each rearrangement, there exists a unique \textbf{reduced} graph pair diagram, where reduced means that it is not the result of an expansion of any other graph pair diagram.

\section{Weak Cell-Transitivity}
\label{SEC transitivity}

Let $G_\mathcal{X}$ be the rearrangement group associated to an \hyperref[TXT expanding]{expanding} replacement system $\mathcal{X} = (X_0, R, \mathsf{C})$, where $X_0$ is the base graph, $R = \{ X_i \mid i \in \mathsf{C} \}$ is the set of replacement graphs and $\mathsf{C}$ is the set of colors.
We denote by $X$ the limit space associated to $\mathcal{X}$.

\begin{definition}
\label{DEF rearrangement transitive}
A subgroup $G$ of a rearrangement group $G_\mathcal{X}$ is \textbf{weakly cell-transitive} when, for each cell $C$ and each proper union of finitely many cells $A$, there exists a $g \in G$ such that $g(A) \subseteq C$.
\end{definition}

As far as the writers know, this class contains most of the rearrangement groups that have been studied so far, along with some notable subgroups.
Indeed, the Basilica rearrangement group $T_B$ and its commutator subgroup $[T_B, T_B]$ (first studied in \cite{Belk_2015}), the Airplane rearrangement group $T_A$ and $[T_A, T_A]$ (\cite{Airplane}) and the rearrangement group of the Vicsek fractal (Example 2.1 of \cite{belk2016rearrangement}) are all examples of weakly cell-transitive rearrangement groups.
Taking $T_A$ as a model, it is easy to argue that for every finite union of cells there exists a rearrangement that maps it into one of the four cells of the base graph.
It can also be seen that there are three orbits of cells (red cells, ``external'' blue cells and ``internal'' blue cells) and that every cell contains an element of each orbit.
Combining the two, one can immediately conclude that $T_A$ is weakly cell-transitive, and a similar argument can be applied to $T_B$, $[T_A, T_A]$, $[T_B, T_B]$ and the Vicsek rearrangement group.

Thompson groups $T$ and $V$ are weakly cell-transitive too, and in truth their actions are actually transitive on the set of cells (standard dyadic intervals).
The same holds for Higman-Thompson groups $T_{r,n}$ and $V_{r,n}$.

As a counterexample, Thompson group $F$ is instead not weakly cell-transitive, since a union of cells containing an endpoint of $[0,1]$ cannot be mapped inside a cell that does not contain said endpoint.
Indeed, in \cite{Gelander2016InvariableGO} it has been shown that $F$ is invariably generated.
For the same reasons, Higman-Thompson groups $F_{r,n}$ are not weakly cell-transitive, although whether these groups are invariably generated is yet to be investigated.

\subsection{Other classes of actions by homeomorphisms}

In this Subsection we explore the relationship between weakly cell-transitive actions and other classes of group actions on topological spaces.

Recall that an action $G \curvearrowright X$ is said to be \textbf{minimal} if the orbit of each point of $X$ is dense in $X$.

\begin{remark}
\label{RMK density}
A subset $D$ of a limit space $X$ is dense in $X$ if and only if it has non-trivial intersection with each cell of $X$.
Indeed, the topological interior of each cell is non-empty, so if $D$ is dense then it must intersect each cell non-trivially.
Conversely, if $A$ is an open subset of $X$, then it must contain some ball, which in turn must contain some cell by \cref{LEM balls cells};
therefore, if $D$ intersects each cell non-trivially then it intersects each open set non-trivially.
\end{remark}

\begin{proposition}
\label{PROP rearrangement orbit dense}
For a subgroup $G$ of a rearrangement group $G_\mathcal{X}$, the following statements are equivalent:
\begin{enumerate}
    \item $G$ is \hyperref[DEF rearrangement transitive]{weakly cell-transitive};
    \item $G$ is \textbf{flexible} (in the sense of \cite{flexible}, although the term was first used in \cite{Belk2019OnTA}), i.e., the group acts on a compact Hausdorff $X$ in such a way that, for any two proper closed sets of $X$ with non-empty topological interior, there is an element that maps one inside the other;
    \item $G$ is \textbf{CO-transitive} (compact-open transitive, \cite{CO-compact}), i.e., for each proper compact $K$ and each non-empty open $U$, there is an element that maps $K$ inside $U$.
\end{enumerate}
Moreover, if $G_\mathcal{X}$ is weakly cell-transitive, then the action of $G_\mathcal{X}$ on the limit space $X$ is minimal.
\end{proposition}

\begin{proof}
We will prove that (1) is equivalent to (2). One can prove the equivalence between (1) and (3) essentially in the same way.
To prove this we will use the fact that each cell contains balls of radius greater than zero (\cref{LEM balls cells}).
First, it is known that the limit space of an \hyperref[TXT expanding]{expanding} replacement system is compact and Hausdorff (Theorem 1.25 of \cite{belk2016rearrangement}).
If $G$ is flexible, $C$ is a cell and $A$ is a proper union of finitely many cells, then $C$ and $A$ are proper closed subsets of the limit space $X$ with non-empty topological interior, so we are done.
Conversely, consider two proper closed subsets $E_1$ and $E_2$ of $X$ with non-empty topological interior.
Since $E_1^\mathsf{C}$ is proper, open and nonempty, it must contain some proper topological interior $\mathring{C_0}$ of a cell $C_{0}$, so $E_1$ is contained in $\mathring{C_0}^\mathsf{C}$.
This is the union of finitely many cells since, for every \hyperref[DEF cellular partition]{cellular partition} containing $C_0$, by \cref{RMK interior} the complement of $\mathring{C_0}$ is the union of every other cell of the partition.
Also, $E_2$ contains $\mathring{E_2}$, which is non-empty, so it contains some cell $C$.
Then, since $G$ is weakly cell-transitive, there exists some $g \in G$ such that $g(\mathring{C_0}^\mathsf{C}) \subseteq C$, thus $g(E_1) \subseteq g(\mathring{C_0}^\mathsf{C}) \subseteq C \subseteq E_2$.

Finally, suppose that $G$ is weakly cell-transitive.
In order to show that the action is minimal, let $p$ be a point of the limit space $X$ and consider a cell $A$ that contains $p$.
By hypothesis, for any cell $C$ there exists a $g \in G$ such that $g(A) \subseteq C$.
Then $g(p) \in C$, so the orbit of $p$ intersects $C$ non-trivially, and by \cref{RMK density} we are done.
\end{proof}

Additionally, by ``forgetting'' the gluing of edges one obtains a totally disconnected space that is homeomorphic to the Cantor set.
Since the action of a rearrangement group is defined on cells, it induces an action on the cones of the Cantor set.
If a subgroup of a rearrangement group is \textbf{vigorous} (\cite{vigorous}) with this action (i.e., for any clopen subset $A$ of the Cantor set, for any two clopen proper subsets $B$ and $C$ of $A$, there is a $g$ in the pointwise stabilizer of the complement of $A$ that maps $B$ inside $C$), then it is weakly cell-transitive.
Indeed, take $A$ as the entire Cantor set: cells correspond to ``glued'' cones, which are clopen.

\begin{figure}
\centering
\begin{subfigure}{.45\textwidth}
\centering
\begin{tikzpicture}
    \draw[->-=.5] (0,0) circle (.75);
    \node at (.75,0) [circle,fill,inner sep=1.5]{};
    
    \draw[->-=.5] (2,0) circle (.75);
    \node at (2.75,0) [circle,fill,inner sep=1.5]{};
\end{tikzpicture}
\caption{The base graph $\Gamma$.}
\end{subfigure}
\begin{subfigure}{.45\textwidth}
\centering
\begin{tikzpicture}
    \draw[-angle 60] (0,0) node[circle,fill,inner sep=1.25]{} node[left]{$v_i$} -- (0,1); \draw (0,1) -- (0,2) node[circle,fill,inner sep=1.25]{} node[left]{$v_t$};
    
    \draw[-stealth] (0.4,1) -- (0.8,1);
    
    \draw[-angle 60] (1.2,0) node[circle,fill,inner sep=1.25]{} node[left]{$v_i$} -- (1.2,0.5); \draw (1.2,0.5) -- (1.2,1) node[circle,fill,inner sep=1.25]{}; \draw[-angle 60] (1.2,1) -- (1.2,1.5); \draw (1.2,1.5) -- (1.2,2) node[circle,fill,inner sep=1.25]{} node[left]{$v_t$};
\end{tikzpicture}
\caption{The replacement rule $e \to R$.}
\end{subfigure}
\caption{A replacement system whose rearrangement group has minimal action but is not weakly cell-transitive.}
\label{FIG double T}
\end{figure}
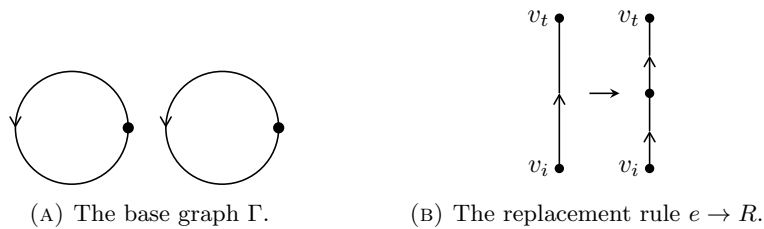

A necessary condition for weak cell-transitivity is that each replacement graph has an expansion that features edges of every type.
The \textit{types} of edges that we refer to are those obtained by partitioning the set of all edges that appear in the expansions of the base graph based on their colors and on whether they are loops or not.
In truth, this condition is very broad, as it is also a necessary condition for the minimality of the rearrangement group.
Indeed, if the replacement graph $R_c$ does not feature edges of a certain type among its expansions, then any point included in some edge of that type cannot be mapped inside a $c$-colored cell.
In general, finding a fully combinatorial necessary or sufficient condition for weak cell-transivitiy seems unlikely, as even replacement systems that do not seem to have special properties may end up producing trivial rearrangement groups (see \cite[Example 2.5]{belk2016rearrangement}).

An example of a rearrangement group whose action is minimal but not weakly cell-transitive is given by the replacement system of \cref{FIG double T}.
Since its base graph is disconnected, it can be seen that its group can be decomposed as the wreath product of $T$ with $S_2$ (the group of permutations of $\{1,2\}$).
In general, it seems that the rearrangement group of every replacement system whose base graph is disconnected can be decomposed in a similar fashion, so it would be interesting to find an example of minimal and not weakly cell-transitive rearrangement group with connected base graph.

\section{Wandering Sets}

As in \cite{Gelander2016InvariableGO}, we introduce the notion of wandering sets, which will be essential in the proof of our \hyperref[THM main]{main Theorem}.

\begin{definition}
Let $G \curvearrowright X$ and $g \in G$. We say that a subset $A \subseteq X$ is:
\begin{itemize}
    \item \textbf{$\bm{g}$-wandering} if, for every $n \in \mathbb{Z}$ such that $g^n \neq 1_G$, we have that $g^n(A) \cap A = \emptyset$;
    \item \textbf{weakly $\bm{g}$-wandering} if, for every $n \in \mathbb{Z}$, either $g^n$ fixes $A$ pointwise or $g^n(A) \cap A = \emptyset$.
\end{itemize} 
\end{definition}

The remainder of this Section is dedicated to proving the following result:

\begin{proposition}
\label{PROP rearrangement wandering}
Let $G_\mathcal{X}$ be rearrangement group. Then each element $g \in G_\mathcal{X}$ admits a weakly wandering cell.
\end{proposition}

This is easy to see for monochromatic rearrangement groups, since they embed nicely inside Higman-Thompson groups $V_{r,n}$, for which the results of \cite{Gelander2016InvariableGO} apply with barely any modification.

In order to prove this result in the general setting of polychromatic rearrangement groups, we will need to generalize the results of \cite{Gelander2016InvariableGO}.
We will distinguish between two cases depending on whether the element of $G_\mathcal{X}$ is periodic or not.
Lemmas \ref{LEM periodic wandering} and \ref{LEM non-periodic wandering} contained in the two following Subsections will thus prove \cref{PROP rearrangement wandering}.

Before we prove \cref{PROP rearrangement wandering}, we point out the following consequence:

\begin{corollary}
Each weakly cell-transitive subgroup of a rearrangement group contains a copy of the free group on two generators, and in particular it is not amenable.
\end{corollary}

This can be easily proved using \cref{COR rearrangement conj} and the ping-pong Lemma.

\subsection{Periodic Elements}

The Lemma below is a generalization of a property of Thompson groups $F$, $T$ and $V$: if an element $g$ fixes some irrational point $\alpha$ of $[0,1]$, then it must fix some small enough neighborhood of $\alpha$.

In the context of limit spaces of replacement systems, we say that a point $p$ is \textbf{irrational} if it has a unique non-periodic \hyperref[TXT gluing relation]{address}.
Observe that irrational points always exist, and in truth the following holds.

\begin{remark}
\label{RMK irrational dense}
The set of irrational points of a limit space is dense.

Indeed, each cell has countably many points with more than one address, as they must be vertices of some expansion of that cell, and it has countably many points with eventually periodic address.
Since \hyperref[TXT gluing relation]{each cell is uncountable}, we can conclude that it contains uncountably many irrational points.
In particular, because of \cref{RMK density}, the set of irrational points is dense.
\end{remark}

\begin{lemma}
\label{LEM irrational fix}
If a rearrangement $g$ fixes an irrational point $p$ of $X$, then there exists a small enough cell $C$ containing $p$ that is fixed by $g$ pointwise.
\end{lemma}

\begin{proof}
Suppose $p$ is an irrational point, and let $e_1 e_2 \cdots$ be its unique address, which must be non-periodic.
Let $(D,R,\sigma)$ be a graph pair diagram for $g$.
Since the point $p$ is irrational, it is included in a unique cell $C_D = C(e_1 \cdots e_n)$ of $D$, which $g$ maps canonically to some cell $C_R$ of $R$.
Since $p$ has a unique address and is contained in $C_R$, we must have $C_R = C(e_1 \cdots e_m)$ for some $m$.
Now, $(D,R,\sigma)$ is a graph pair diagram for $g$, so the action on the cells of $D$ is canonical, meaning that the effect of $g$ on the points represented by addresses $e_1 \cdots e_n w$ is the replacement of the prefix $e_1 \cdots e_n$ with $e_1 \cdots e_m$.
In particular, $p = g(p)$ is represented by the addresses $e_1 e_2 \cdots e_n e_{n+1} \cdots$ and $e_1 \cdots e_m e_{n+1} \cdots$.
Since $p$ has unique address, this implies that $e_1 \cdots e_n e_{n+1} \cdots$ is the same as $e_1 \cdots e_m e_{n+1} \cdots$.
This is only possible if either $e_1 \cdots e_n = e_1 \cdots e_m$ or $e_{n+1} e_{n+2} \cdots$ is periodic, but the latter is impossible since $p$ is non-periodic.
\end{proof}

\begin{lemma}
\label{LEM periodic wandering}
If $g \in G_\mathcal{X}$ is periodic, then there exists a weakly $g$-wandering cell.
\end{lemma}

\begin{proof}
Suppose $g$ is a non-trivial rearrangement.
Then, because of \cref{RMK irrational dense}, there must exist an irrational point $p$ that is not fixed by $g$.
Let $n$ be the period of $p$ under the action of $g$.
Since $g, g^2, \dots, g^{n-1}$ are homeomorphisms of $X$ and $X$ is a Hausdorff topological space (Theorem 1.25 of \cite{belk2016rearrangement}) there exist a small enough ball $B$ of $p$ such that $g(B), \dots, g^{n-1}(B)$ are disjoint from $B$.
By \cref{LEM irrational fix}, we can assume that $g^n$ fixes $B$ pointwise.
Then, by \cref{LEM balls cells}, there exists a cell contained in $B$, and so that cell is weakly $g$-wandering.
\end{proof}

\subsection{Non-periodic Elements}

Following the strategy of \cite{Gelander2016InvariableGO}, we need to gather some results about the dynamics of rearrangements in order to prove the existence of wandering cells for non-periodic elements.
The following Lemmas are generalizations of Lemmas 10.2, 10.3 and 10.5 of \cite{Brin2004HigherDT}, along with the general structure of its Section 10, which is about \textit{revealing pairs} of trees for elements of $V$.

Observe that, following the same idea of addresses described at \cpageref{TXT gluing relation}, if $E$ is an expansion of the base graph $\Gamma_0$, then we can represent the expansions performed to obtain $E$ from $\Gamma_0$ as a forest $F_E$, where each root represents an edge of $\Gamma_0$ and each caret represents an edge expansion.
By \textit{caret} here we mean an interior node $\nu$ of the forest along with all of its children, the amount of which depends on the color of $\nu$ (in fact, we may have non-regular rooted forests when dealing with polychromatic replacement systems).
It is clear that each simple expansion of $E$ corresponds uniquely to appending a specific tree, which depends on the color of the edge being expanded, to the forest $F_E$ at the expanded edge.

Consider a graph pair diagram $(D,R,\sigma)$ and denote by $F_D$ and $F_R$ the forests corresponding to $D$ and $R$, respectively.
We can compute $F_D - F_R$ (and $F_R - F_D$), by which we mean the set of carets that belong to $F_D$ but not to $F_R$.
These correspond one-to-one to expansions performed in $D$ that have not been performed in $R$.
We call \textbf{domain imbalance} the number of carets of $F_D - F_R$.

\begin{remark}
\label{RMK imbalances}
For polychromatic replacement systems, this need not be the same as the number of carets of $F_R - F_D$, which we could call \textit{range imbalance}.
This is no issue since, among all of the representatives for the same rearrangement, domain imbalance and range imbalance always differ by the same number.

Indeed, representatives for the same rearrangements differ by sequences of simple expansions or reductions of domain and range, so it suffices to see how imbalances change when going from $(D,R,\sigma)$ to a simple expansion $(D \triangleleft e, R \triangleleft \sigma(e), \sigma')$, which consists of appending a caret $U$ at $e$ in $F_D$ and the same caret $U$ at $\sigma(e)$ in $F_R$.
If $e = \sigma(e)$, clearly the imbalances remain the same, as the added carets cancel out.
If $e \neq \sigma(e)$, we distinguish cases based on whether $e$ is an internal node of $R$ and whether $\sigma(e)$ is an internal node of $D$:
\begin{enumerate}
    \item if both happen, then the simple expansion decreases domain imbalance and range imbalance by one;
    \item if only one happens, then both the domain and the range imbalance are left unchanged by the simple expansion;
    \item if neither of the two happen, then the simple expansion increases both domain and range imbalance by one.
\end{enumerate}
In all of these cases, domain and range imbalance change by the same amount.
\end{remark}

When computing $F_D - F_R$, we can also consider its \textbf{components}, by which we mean the pairwise disjoint ``trees'' in which $F_D$ breaks up after carets of $F_R$ have been removed.
Each component represents a maximal set of expansions that are ``not independent'', by which we mean expansions of edges $e_1, \dots, e_k$ where $D = E \triangleleft e_1 \triangleleft \dots \triangleleft e_k$ such that, for $i \geq 2$, each $e_i$ is not an edge of $E$ but is instead created by subsequent expansions.

\phantomsection\label{TXT minimality choice}
Assume that $(D,R,\sigma)$ is a graph pair diagram with the least domain imbalance of all of the representatives of the same rearrangement (which is the same as assuming it has least range imbalance, by \cref{RMK imbalances}).
Among the representatives with minimal domain imbalance, we choose one that has the smallest number of components of $F_D-F_R$.
Among all such representatives, we choose one that has the smallest number of components of $F_R-F_D$.
In the remainder of this Subsection we will prove facts about graph pair diagrams $(D,R,\sigma)$ chosen in this manner.

Before we start with the first Lemma, consider a sequence $u_1, \dots, u_n$ of edges of $D$ such that $u_1, \dots, u_n$ and $\sigma(u_n)$ are pairwise distinct and such that $u_{i+1} = \sigma(u_i)$ for all $1 \leq i < n$.
We call this an \textbf{expandable sequence} of edges of $D$ (leaves of $F_D$).
Observe that, in particular, these edges must have the same color, and that $u_2, \dots, u_n, \sigma(u_n)$ must be edges of $R$.
We can then perform the same expansion (not necessarily a simple one) on each of $u_1, \dots, u_n$ in $D$ and each of $\sigma(u_1), \dots, \sigma(u_n)$ in $R$, obtaining a new graph pair diagram that represents the same rearrangement.
We call this operation an \textbf{iterated expansion}.
Observe that, when looking at the pair of forests $F_D$ and $F_R$, this operation corresponds exactly to attaching the same subtree (representing the expansion) to the leaves $u_1, \dots, u_n$ in $D$ and to the leaves $\sigma(u_1), \dots, \sigma(u_n)$ in $R$.

\begin{lemma}
\label{LEM Brin 1}
There is no expandable sequence $u_1, \dots, u_n$ such that $u_1$ is an interior node of $F_R$ and $\sigma(u_n)$ is an interior node of $F_D$.
\end{lemma}

\begin{proof}
Suppose by contradiction that the statement does not hold and consider the component $U$ of $F_D - F_R$ with root at $\sigma(u_n)$.
We can perform an iterated expansion by $U$ along $u_1, \dots, u_n$, which causes the following modifications on $F_D - F_R$:
\begin{itemize}
    \item the copy of $U$ appended at $\sigma(u_n)$ is removed;
    \item the copies of $U$ appended to both $F_D$ and $F_R$ at each $u_i$ for $2 \leq i \leq n$ do not contribute to $F_D - F_R$, as they cancel out;
    \item since $u_1$ is an interior node of $F_R$, the number of carets added to $F_D$ at $u_1$ from the copy of $U$ is strictly less than the number of carets of $U$.
\end{itemize}
This would lower the domain imbalance, which is not possible by our \hyperref[TXT minimality choice]{choice} of minimality of domain imbalance.
\end{proof}

\begin{lemma}
\label{LEM Brin 2}
There is no expandable sequence $u_1, \dots, u_n$ such that:
\begin{enumerate}
    \item $u_1$ is not a node of $F_R$,
    \item $\sigma(u_n)$ is an interior node of $F_D$,
    \item the component of $F_D - F_R$ rooted at $\sigma(u_n)$ and the one containing $u_1$ are not the same component.
\end{enumerate}
\end{lemma}

\begin{proof}
Suppose by contradiction that there is such an expandable sequence $u_1, \dots, u_n$.
Let $U$ and $V$ respectively be the components of $F_D - F_R$ rooted at $\sigma(u_n)$ and containing $u_1$, and consider an iterated expansion by $U$ along the sequence $u_1, \dots, u_n$.
This iterated expansion does not change the domain imbalance, but it removes the component $U$ from $F_D - F_R$, while introducing no new component to $F_D - F_R$.
Indeed, the subtrees appended in both $F_D$ and $F_R$ to $u_i$ for $2 \leq i \leq n$ cancel out in $F_D - F_R$ and a copy of $U$ is appended at the already existing component $V$.
This is not possible by our \hyperref[TXT minimality choice]{choice} of minimal number of components of $F_D - F_R$ among those with minimal domain imbalance.
\end{proof}

\begin{lemma}
\label{LEM Brin 3}
There is no expandable sequence $u_1, \dots, u_n$ such that:
\begin{enumerate}
    \item $\sigma(u_n)$ is not a node of $F_D$,
    \item $u_1$ is an interior node of $F_R$,
    \item the component of $F_R - F_D$ rooted at $u_1$ and the one containing $\sigma(u_n)$ are not the same component.
\end{enumerate}
\end{lemma}

\begin{proof}
This proof is dual to that of the previous Lemma.
Suppose by contradiction that there is such an expandable sequence $u_1, \dots, u_n$.
Let $U$ and $V$ respectively be the components of $F_R - F_D$ containing $\sigma(u_n)$ and rooted at $u_1$ and consider an iterated expansion by $V$ along the sequence $u_1, \dots, u_n$.
This iterated expansion does not change the domain imbalance, but it removes the component $V$ from $F_R - F_D$, while introducing no new component to $F_R - F_D$.
Indeed, the subtrees appended in both $F_D$ and $F_R$ to $u_i$ for $2 \leq i \leq n$ cancel out in $F_R - F_D$ and a copy of $V$ is appended at the already existing component $U$.
This is not possible by our \hyperref[TXT minimality choice]{choice} of minimal number of components of $F_R - F_D$ among those with minimal number of components of $F_D - F_R$, which in turn were chosen among those with minimal domain imbalance.
Note that the argument works because the number of components of $F_D - F_R$ has not increased, since $\sigma(u_n)$ is not a node of $F_D$ by assumption.
\end{proof}

\begin{lemma}
\label{LEM Brin 4}
\begin{enumerate}
    \item []
    \item For each non-trivial component $U$ of $F_D - F_R$, there is a unique leaf $l = l(U)$ of $U$ such that there is an expandable sequence $u_1, \dots, u_n$ starting at $u_1 = l$ and ending at $\sigma(u_n) = r$, where $r = r(U)$ denotes the root of $U$.
    \item For each non-trivial component $V$ of $F_R - F_D$, there is a unique leaf $l = l(V)$ of $V$ such that there is an expandable sequence $u_1, \dots, u_n$ starting at $u_1 = r$ and ending at $\sigma(u_n) = l$, where $r = r(V)$ denotes the root of $V$.
\end{enumerate}
\end{lemma}

\begin{proof}
We will only prove the second point of the Lemma, as the first point is proven in a very similar manner, using \cref{LEM Brin 2} in place of \cref{LEM Brin 3}.

Observe that $r$ is a leaf of $F_D$, so let $u_1 = r$.
This can be regarded as a trivial expandable sequence.
Since $u_1 = r$ is an interior node of $F_R$, by \cref{LEM Brin 1} we have that $\sigma(u_1)$ cannot be an interior node of $F_D$.
For the same reason, by \cref{LEM Brin 3} we have that either $\sigma(u_1)$ is a node of $F_D$ or it belongs to the component $V$ of $F_R - F_D$.
Thus, $\sigma(u_1)$ is either a leaf of $F_D$ or it is a leaf of $V$, and only one of these is possible.

Now, assume that $u_1, \dots, u_n$ is an expandable sequence with $r = u_1$ and with $\sigma(u_n)$ either a leaf of $V$ or a leaf of $F_D$.
There is at least one such expandable sequence, as shown in the previous paragraph.
If $\sigma(u_n)$ is a leaf of $V$, we are done, so assume $\sigma(u_n)$ is a leaf of $F_D$ instead.
We can then create a longer expandable sequence by adding the element $u_{n+1} = \sigma(u_n)$ at the end of the sequence.
This new sequence satisfies the same property of the previous one: since $u_1 = r$ is not an interior node of $F_R$, by \cref{LEM Brin 3} we have that $\sigma(u_n) = u_{n+1}$ can either be a common leaf of $F_D$ and $F_R$ or a leaf of $V$.

This expandable sequence cannot extend infinitely.
Indeed, if this were not the case, since there is a finite amount of leaves of $F_D$ among which we can choose and $g$ is an invertible map, for some $n$ we would have $u_1 = \sigma(u_n)$.
But $u_1$ would be a leaf of $F_R$, whereas it should be an interior node of $F_R$.
This means that we must reach a leaf $l$ with the property we are looking for, so we are done.
\end{proof}

It follows from the second point of \cref{LEM Brin 4} that, for each non-periodic rearrangement $g$, if we choose a graph pair diagram $(D,R,\sigma)$ \hyperref[TXT minimality choice]{as described earlier}, then for each leaf $r$ of $F_D$ that is an interior node in $F_R$, there is a leaf $l$ of $F_R$ located below $r$ and an $n \in \mathbb{N}^* = \{1, 2, \dots \}$ such that $r, g(r), \dots, g^{n-1}(r)$ are distinct leaves of $F_D$ and $g^n(r) = l$.
In terms of graph pair diagrams, this means that for each edge $e$ of $D$ that has been expanded in $R$, there is an edge $e^*$ of $R$ generated by an expansion of $e$ and an $n \in \mathbb{N}^*$ such that the topological interiors of the cells generated by the edges $e, g(e), \dots, g^{n-1}(e)$ of $D$ are such that $g^n(e) = e^*$ and they are pairwise disjoint (observe that boundary points do not interfere with this because of \cref{RMK interior}).
Essentially, given the topological interior $\mathring{C}$ of a cell whose generating edge is in $D$ and has been expanded in $R$, iterating $g$ maps $\mathring{C}$ to disjoint topological interiors of cells for $n-1$ steps, and then to a topological interior contained inside $\mathring{C}$ itself at the $n$-th step.

\begin{lemma}
\label{LEM non-periodic wandering}
If $g \in G_\mathcal{X}$ is non-periodic, then there exists a $g$-wandering cell.
\end{lemma}

\begin{proof}
Let $(D,R,\sigma)$ be a graph pair diagram for $g$ as discussed above.
Since $g$ is non-periodic, $D \neq R$, so there exists some edge $e$ of $D$ that has been expanded in $R$.
Consider the edge $e^*$ and the number $n \in \mathbb{N}^*$ given by the discussion above.
Let $f$ be an edge of $R$ that is distinct from $e^*$ and that has been generated by an expansion of $e$, so $C(f) \subseteq C(e)$.
We claim that the topological interior $\mathring{C}(f)$ is $g$-wandering, which would conclude the proof, as it implies that every cell contained in $\mathring{C}(f)$ is also $g$-wandering.

Indeed, let $0 \neq m \in \mathbb{Z}$.
For every $U \subseteq X$, if $g^m(U) \cap U \neq \emptyset$ then $U \cap g^{-m}(U) \neq \emptyset$, so we can assume that $m > 0$.
Let $m = qn + r$, where $q \geq 0$ and $r \in \{0, 1, \dots, n-1\}$.
Since $g^n(\mathring{C}(e)) = \mathring {C}(e^*) \subseteq \mathring{C}(e)$, we have $g^{qn}(\mathring{C}(e)) \subseteq \mathring{C}(e^*)$, and so $g^{qn}(\mathring{C}(f)) \subseteq \mathring{C}(e^*)$.

Now, if $r=0$ we are done, as $f$ and $e^*$ being distinct edges of $D$ implies $\mathring{C}(f) \cap \mathring{C}(e^*) = \emptyset$.
If instead $r \neq 0$, then $g^m(\mathring{C}(f)) = g^{qn + r} (\mathring{C}(f)) \subseteq g^r(\mathring{C}(e))$, which is disjoint from $\mathring{C}(e)$ because $r \in \{1, \dots, n-1\}$.
Thus, since $\mathring{C}(f) \subseteq \mathring{C}(e)$ we have $g^m(\mathring{C}(f)) \cap \mathring{C}(f) = \emptyset$, so we are done.
\end{proof}

\section{Proof of the Theorem}

The results of the previous Section give us the following Corollary:

\begin{corollary}
\label{COR rearrangement conj}
Suppose that $G$ is a weakly cell-transitive subgroup of some rearrangement group $G_\mathcal{X}$.
Let $g \in G$ and let $A \subset X$ be a union of finitely many cells of $\mathcal{X}$. Then there exists some $h \in G$ such that $A$ is weakly $g^h$-wandering.
\end{corollary}

\begin{proof}
Because of \cref{PROP rearrangement wandering}, there exists a weakly $g$-wandering cell $I \subset X$.
Since $G$ is weakly cell-transitive, there exists an $h \in G$ such that $A \subseteq h^{-1} (I)$.
Now, it is clear that $h^{-1} (I)$ is weakly $g^h$-wandering, so $A$ is too and we are done.
\end{proof}

\begin{theorem}
\label{THM main}
Every weakly cell-transitive subgroup $G$ of a rearrangement group $G_\mathcal{X}$ is not invariably generated.
\end{theorem}

Note that this is equivalent to the \hyperref[THM main intro]{Main Theorem} stated in the \hyperref[SEC intro]{Introduction} because of \cref{PROP rearrangement orbit dense}.

\begin{proof}
Observe that there are at most countably many non-trivial conjugacy classes of $G$, since every rearrangement group $G_\mathcal{X}$ is itself at most countably infinite
(indeed, each rearrangement can be represented as a graph isomorphism of some full expansion graph $E_n$, and for each $n \in \mathbb{N}$ there are finitely many such isomorphisms).
First, assume that there are exactly countably many conjugacy classes, so let $\{Cl_n\}_{n \in \mathbb{N}^*}$ be the non-trivial conjugacy classes of $G$, where $\mathbb{N}^* = \{1, 2, \dots \}$.

Consider a point $p$ of the limit space $X$.
Recall that each point of the limit space is the equivalence class under the gluing relation of at least one element of the space of infinite addresses, which is an infinite sequence of cells decreasing with respect to the inclusion.
Let $C_0 C_1 C_2 \cdots$ be such a sequence for the point $p$.
For each $n \in \mathbb{N}^* = \{ 1, 2, \dots \}$, let $I_n$ be the set difference of the topological interior of $C_n$ with the cell $C_{n+1}$.
We intentionally do not consider $n = 0$ in order to make sure that there is at least some cell that is not contained in the union of these $I_n$, as the cell $C_0$ is the whole limit space whenever the base graph consists solely of an edge.
Note that these $I_n$'s are pairwise disjoint and converge to the point $p$.

Now, for each $n \in \mathbb{N}^*$, the complement $I_n^{\mathsf{C}}$ is a non-trivial union of finitely cells.
More explicitly, if $E$ is an expansion containing the edge $e$ corresponding to $C_n$, then consider $E \triangleleft e$: the set $I_n^{\mathsf{C}}$ is the union of all cells of $E \triangleleft e$ that also appeared in $E$ along with $C_{n+1}$.
Hence, by \cref{COR rearrangement conj} there exists a $\gamma_n$ in the $n$-th conjugacy class $Cl_n$ such that $I_n^{\mathsf{C}}$ is weakly $\gamma_n$-wandering.

We claim that the orbit of $p$ under the action of the subgroup $H = \langle \gamma_{n} \mid n \in \mathbb{N}^* \rangle$ of $G$ is contained in $I = \bigcup_{n \in \mathbb{N}^*} I_n$.
Then, since there is at least a cell that is not contained in $I$, the action of $H$ is not minimal.
So, by virtue of \cref{PROP rearrangement orbit dense}, we would get $H \neq G$ as required.
Thus, \cref{LEM rearrangement ping-pong} below completes the proof in the case of countably many conjugacy classes of $G$.

If there are only finitely many conjugacy classes, then we can do everything in exactly the same way, except that we will have a finitely generated $H = \langle \gamma_1, \dots, \gamma_k \rangle$ under whose action the orbit of $p$ is contained in the union of finitely many $I_n$'s, leading to the same conclusion for the same \cref{LEM rearrangement ping-pong}.
\end{proof}

The following Lemma and its proof are the same as Lemma 18 of \cite{Gelander2016InvariableGO}. The only difference is that it applies to cells of a limit space instead of intervals of $S^1$.

\begin{lemma}
\label{LEM rearrangement ping-pong}
Let $q$ be an element of the $H$-orbit of $p$ that is different from $p$.
Let $g \in H$ be an element of minimal word-length over the alphabet $A := \{ \gamma_n \mid n \in J \}$ such that $g(p) = q$ (where $J$ is either $\mathbb{N}$ or $\{1, \dots,, k \}$).
Assume that $g = \gamma_{i_1}^{k_1} \cdots \gamma_{i_m}^{k_m}$, where:
\begin{itemize}
    \item $k_1, \dots, k_m \neq 0$,
    \item $i_{j+1} \neq i_j$ for each $j=1, \dots, m-1$,
    \item $|k_1| + \dots + |k_m|$ is the word-length of $g$ over $A$.
\end{itemize}
Then $q \in  I_{i_m}$.
\end{lemma}



As a final remark, it is easy to see that the proofs of the previous \hyperref[COR rearrangement conj]{Corollary} and \hyperref[THM main]{Theorem} can also be applied to normal subgroups of a weakly cell-transitive subgroup of a rearrangement group, providing the following more general statement:

\begin{proposition}
\label{PROP main}
If $1 \neq N \trianglelefteq G \leq G_\mathcal{X}$, where $G_\mathcal{X}$ is a rearrangement group and $G$ is weakly cell-transitive, then $N$ is not invariably generated.
\end{proposition}

In particular, we have the following notable result:

\begin{corollary}
\label{COR commutator}
The commutator subgroup of a weakly cell-transitive rearrangement group is not invariably generated.
\end{corollary}

\section*{Acknowledgements}

The authors gratefully acknowledge Francesco Matucci for providing many insights and suggestions and they would like to thank the two anonymous referees for several helpful comments.
The authors are also thankful to Jim Belk and Collin Bleak for useful conversations and advice, and are grateful to Jim Belk and Bradley Forrest for kindly providing the image of the Airplane limit space (\cref{FIG A limit space}) from their work \cite{belk2016rearrangement}.
This same image was also colored by the authors to produce \cref{fig_cells_A} and \cref{fig_action_alpha}.

\printbibliography[heading=bibintoc]

\end{document}